\documentclass[aps,pra,10pt,onecolumn,floatfix,nofootinbib]{revtex4-2}
\usepackage[utf8]{inputenc}
\usepackage[english]{babel}
\usepackage[T1]{fontenc}
\usepackage{amsmath,amsthm,amsfonts,amssymb,times,bbm,graphicx,xcolor}
\usepackage{physics}
\usepackage{hyperref}
\hypersetup{citecolor=blue}
\hypersetup{colorlinks=true}
\hypersetup{linkcolor=blue}
\hypersetup{urlcolor=blue}
\newtheorem{theorem}{Theorem}

\newtheorem{lemma}[theorem]{Lemma}

\newtheorem{definition}[theorem]{Definition}

\def\RR{\mathbbm{R}}

\def\NN{\mathbbm{N}}

\def\Id{\mathbbm{1}}

\DeclareMathOperator{\id}{id}
\DeclareMathOperator{\conv}{conv}

\DeclareMathOperator{\Prob}{Prob}

\DeclareMathOperator{\Ext}{Ext}
\DeclareMathOperator{\Sym}{Sym}

\newcommand{\tmin}{\dot{\otimes}}
\newcommand{\tmax}{\hat{\otimes}}

\def\S{\mathcal{S}}

\def\NN{\mathbb{N}}
\def\RR{\mathbb{R}}

\begin{document}
\title{The polarization hierarchy for polynomial optimization over convex bodies,\\with applications to nonnegative matrix rank}

\author{Martin Plávala}
\affiliation{Naturwissenschaftlich-Technische Fakult\"{a}t, Universit\"{a}t Siegen, Walter-Flex-Stra\ss e 3, 57068 Siegen, Germany}
\author{Laurens T.\ Ligthart}
\author{David Gross}
\affiliation{Institute for Theoretical Physics, University of Cologne, Germany }
\date{June 13th, 2024}

\begin{abstract}
	We construct a convergent family of outer approximations for the problem of optimizing
	polynomial functions over convex bodies 
	subject to polynomial constraints.
	This is achieved by generalizing the \emph{polarization hierarchy}, which has previously been introduced for the study of polynomial optimization problems over state spaces of $C^*$-algebras, to convex cones in finite dimensions.
	If the convex bodies can be characterized by linear or semidefinite programs, then the same is true for our hierarchy.
	Convergence is proven by relating the problem to a certain \emph{de Finetti theorem} for \emph{general probabilistic theories}, which are studied as possible generalizations of quantum mechanics.
	We apply the method to the problem of nonnegative matrix factorization, and in particular to the \emph{nested rectangles problem}.
	A numerical implementation of the third level of the hierarchy is shown to give rise to a very tight approximation for this problem.
\end{abstract}

\maketitle

\section{Introduction}
Solving practical linear optimization problems is in many cases rather straightforward: we can often use linear or semidefinite programming, or in general convex optimization \cite{boyd2004convex} to arrive to a numerical solution. Such methods are regularly used in quantum information to obtain numerical predictions for subsequent experiments \cite{cavalcanti2016quantum,tavakoli2023semidefinite,skrzypczyk2023semidefinite,mironowicz2023semi}. There are also many problems in quantum information and other fields that are non-linear, for example the separability of quantum states \cite{horodecki2009quantum,guhne2009entanglement}, transformation of quantum gates \cite{navascues2018resetting,quintino2019probabilistic,miyazaki2019complex,bavaresco2021strict}, but also determining nonnegative matrix rank \cite{cohen1993nonnegative,wang2012nonnegative} which was also recently used to analyze quantum communication \cite{heinosaari2020communication,heinosaari2024simple}.

In this paper we aim to provide a method to solve a general class of polynomial optimization problems by deriving a convergent hierarchy of convex relaxations, called the \emph{polarization hierarchy}. This hierarchy is inspired by a similarly named method for optimization of polynomials in states on (non-commutative) $C^*$-algebras constructed in Refs.~\cite{ligthart2022inflation, ligthart2023convergent}. More precisely, if $K_1, \ldots, K_m$ are compact convex subsets of finite dimensional vector spaces, we optimize a multivariate polynomial objective function $p: K_1 \times \ldots \times K_m \to \RR$ under a set of $q$ multivariate polynomial equality constraints given by $f: K_1 \times \ldots \times K_m \to \RR^q$, resulting in an optimization problem of the form
\begin{align}
\begin{split} \label{eq:opt-problem-intro}
p^* = \min_{\{ x_{i} \in K_{i} \}} \quad & p( \{x_{i} \}_{i=1}^m) \\
\text{s.t.} \quad & f(\{x_{i}\}_{i=1}^m)) = 0.
\end{split}
\end{align}
Such optimization problems in particular include many of the aforementioned non-linear problems in quantum information theory and in General Probabilistic Theories (GPTs) \cite{plavala2023general}. If the convex compact set is a polyhedron (e.g.~classical probability theory, or the no-signaling polytope), the polarization method provides a convergent linear programming (LP) hierarchy, while if it is spectrahedral (e.g.~quantum theory), it yields a convergent semidefinite programming (SDP) hierarchy. For more general cones, one can still write down a convergent hierarchy, though these might not be efficiently implementable.

As an example of an application that is not directly motivated by GPTs, let us consider the nonnegative matrix rank: given an $n \times m$ matrix $A$ with non-negative entries we say that the nonnegative matrix rank of $A$ is $k$ if there are matrices $L$ and $R$ with non-negative entries where $L$ is an $n \times k$ matrix and $R$ is a $k \times m$ matrix such that $A = LR$, and there is no smaller $k$ for which this is true. Thus one can state this problem as finding matrices $L$ and $R$ of given sizes and with positive entries such that
\begin{equation} \label{eq:intro-fLRA}
f(L,R) = LR - A = 0.
\end{equation}
Moreover, $A$ and also $L$ and $R$ can, without loss of generality, be considered to be left-stochastic matrices, which will turn out to be important since the set of left-stochastic matrices is compact and convex. It thus follows that the problem at hand can be formulated as follows: given two compact convex sets $K_A$, $K_B$, and an affine function $f: K_A \times K_B \to V$ where $V$ is a vector space, find $x_A \in K_A$ and $x_B \in K_B$ such that $f(x_A, x_B) = 0$. In this example of nonnegative matrix factorization, $K_A$ and $K_B$ are the respective sets of left-stochastic matrices and the function $f$ is given by \eqref{eq:intro-fLRA}. Since it is a feasibility problem, the objective function is simply given by $p = 0$.

In this paper we will show that such problems can be solved using a hierarchy of outer approximation which can be computed using convex optimization. To improve readability, we will focus on the case of two compact convex sets, $K_A$ and $K_B$, and quadratic constraints. The generalization to more sets and higher order polynomials will follow straightforwardly. We will then apply our hierarchies to the nested rectangles problem \cite{gribling2019lower} to show that already at low, finite level we obtain relatively fast convergence for this particular problem. 

\section{polarization hierarchy} \label{sec:polarization_hierarchy}
We will assume that all vector spaces are finite-dimensional. Let $K$ be a \emph{state space}, that is let $K$ be a compact convex subset of a finite-dimensional vector space. We will use $A(K)$ to denote the sets of affine functions on $K$, we will use $A(K)^+$ to denote the set of positive affine functions and we will use $A(K)^*$ to denote the dual vector space of $A(K)$. Note that there is a natural embedding of $K$ into $A(K)^*$ using the evaluation map, i.e., let $x \in K$ then we can define $\varphi_x \in A(K)^*$ as follows: let $f \in A(K)$, then $\varphi_x(f) = f(x)$. There is also a distinguished constant function $1_K \in A(K)$ defined as $1_K(x) = 1$ for all $x \in K$, see \cite{plavala2023general} for a detailed construction and proofs.

Let $K_A$ and $K_B$ be state spaces and let $V$ be a finite-dimensional vector space. We will use the shorthand $1_A = 1_{K_A}$ and $1_B = 1_{K_B}$ to simplify the notation. Our task is, given an affine function $f: K_A \times K_B \to V$, to determine whether there are $x_A \in K_A$ and $x_B \in K_B$ such that $f(x_A, x_B) = 0$. Since $f$ is affine, this includes also problems of the form $f(x_A, x_B) = a$ for any $a \in V$. Also note that affine functions $K_A \times K_B \to V$ are in one-to-one correspondence to linear maps $A(K_A)^* \otimes A(K_B)^* \to V$, we will thus consider $f$ to be a linear map $f: A(K_A)^* \otimes A(K_B)^* \to V$ \cite{ryan2002introduction}. We can thus rewrite the problem as the task of finding $x_A \in K_A$ and $x_B \in K_B$ such that $f(x_A \otimes x_B) = 0$.

Let $K_A \tmin K_B$ denote the set of all \emph{separable} states of $K_A$ and $K_B$, that is,
\begin{equation}
K_A \tmin K_B = \conv(\{ x_A \otimes x_B : x_A \in K_A, x_B \in K_B \}),
\end{equation}
we will also denote by $K_A \tmax K_B$ the maximal tensor product of $K_A$ and $K_B$ defined as
\begin{equation}
K_A \tmax K_B = \{ \varphi \in A(K_A)^* \otimes A(K_B)^* : \varphi(f_A \otimes f_B) \geq 0, \forall f_A \in A(K_A)^+, \forall f_B \in A(K_B)^+, \varphi(1_{K_A} \otimes 1_{K_B}) = 1 \}.
\end{equation}
The natural relaxation of our problem is to find $y_{AB} \in K_A \tmin K_B$ such that $f(y_{AB}) = 0$; one can solve this problem using the hierarchy developed in \cite{aubrun2022monogamy}. This is only a relaxation since $y_{AB} \in K_A \tmin K_B$ implies that $y_{AB} = \sum_i \lambda_i x_{i,A} \otimes x_{i,B}$ and we get $\sum_i \lambda_i f(x_{i,A} \otimes x_{i,B}) = 0$, but in general $f(x_{i,A} \otimes x_{i,B}) \neq 0$ in this case. We will use the polarization hierarchy \cite{ligthart2022inflation} to make this a reformulation of the problem rather than relaxation. To do so we will require a linear map $\Pi: V \otimes V \to W$, where $W$ is a yet unspecified vector space, such that $\Pi(a \otimes a + b \otimes b) = 0$ implies $a = b = 0$. Note that such a linear map always exists as we can always take $W = V \otimes V$ and $\Pi(a \otimes b) = a \otimes b$, but other choices are also possible. Then observe that $f(x_A \otimes x_B) = 0$ is equivalent to requiring that $\Pi(f(x_A \otimes x_B) \otimes f(x_A \otimes x_B)) = 0$ and we have
\begin{equation}
0 = \Pi(f(x_A \otimes x_B) \otimes f(x_A \otimes x_B)) = \Pi((f \otimes f)(x_A \otimes x_B \otimes x_A \otimes x_B)).
\end{equation}
This insight is the main building block necessary to formulate the polarization hierarchy. In order to do so, we will also need the concept of a bicompatible sequence, which is defined as follows:
\begin{definition}
Let $K_A$ and $K_B$ be state spaces. We say that $\{y_n\}_{n \in \NN}$ is a bicompatible sequence if $y_n \in K_A^{\tmax n} \tmax K_B^{\tmax n}$ such that $y_{n-1} = (1_{A_n} \otimes 1_{B_n})(y_n)$ where $1_{A_n}$, $1_{B_n}$ denote the distinguished functions that act as a \emph{partial trace} over the $n$\textsuperscript{th} copy of $K_A$, $K_B$, respectively, that is, $1_{A_n}(\otimes_{i=1}^n v_{i, A}) = 1_A(v_{n, A}) \otimes_{i=1}^{n-1} v_{i, A}$ for any $v_{i, A} \in A(K_A)^*$ and analogously for $1_{B_n}$, and
\begin{equation} \label{eq:symmetry}
y_n = (S_{\sigma_A} \otimes S_{\sigma_B})(y_n)
\end{equation}
where $S_{\sigma_A}$, $S_{\sigma_B}$ is the linear map that permutes the tensor factors of $K_A^{\tmax n}$, $K_B^{\tmax n}$, respectively.
\end{definition}

One can always construct bicompatible sequences of the form $y_n = \sum_i \lambda_i x_{A,i}^{\otimes n} \otimes x_{B,i}^{\otimes n}$ for some fixed $x_{A,i} \in K_A$, $x_{B,i}^{\otimes n} \in K_B$ and $\lambda_i \in [0,1]$, $\sum_i \lambda_i = 1$. We will need to prove that the converse also holds:
\begin{lemma} \label{lemma:structureYn}
Let $\{y_n\}_{n \in \NN}$ be a bicompatible sequence, then there is an essentially unique Borel probability measure $\mu \in \Prob(K_A \times K_B)$ such that for every $n\in\NN$,
\begin{equation}
y_n = \int_{K_A \times K_B} x_A^{\otimes n} \otimes x_B^{\otimes n} \dd \mu(x_A, x_B).
\end{equation}
\end{lemma}
\begin{proof}
The result follows using the de Finetti theorem in the form as described in \cite[Remark 1.]{aubrun2022monogamy}. We provide an explicit proof that relies on the de Finetti theorem for compact convex sets \cite{barrett2009finetti}.

Assume that $\{y_n\}_{n \in \NN}$ is a bicompatible sequence, then $y_n \in K_A^{\tmax n} \tmax K_B^{\tmax n}$, but we can, up to an isomorphism that permutes the tensor factors, also treat $y_n$ as an element of $(K_A \tmax K_B)^{\tmax n}$. We will denote this isomorphism $\S_n$, that is, $\S_n: K_A^{\tmax n} \tmax K_B^{\tmax n} \to (K_A \tmax K_B)^{\tmax n}$ is just the reordering of the tensor factors defined as
\begin{equation}
\S_n((\otimes_{i=1}^n x_{i,A}) \otimes (\otimes_{i=1}^n x_{i,B}) ) = \otimes_{i=1}^n (x_{i,A} \otimes x_{i,B})
\end{equation}
for $x_{i,A} \in K_A$ and $y_{i,B} \in K_B$ for all $i \in \{1,\ldots,n\}$.

Observe now that $\{ \S_n(y_n) \}_{n \in \NN}$ is a compatible sequence over $K_A \tmax K_B$, that is, $\S_n(y_n)$ is invariant under permutations acting on the $n$ copies of $K_A \tmax K_B$ and $(1_{A_n} \otimes 1_{B_n})(\S_n(y_n)) = \S_{n-1}(y_{n-1})$, where we are using $1_{K_A \tmin K_B} = 1_{K_A \tmax K_B} = 1_A \otimes 1_B$. According to the de Finetti theorem for compact convex sets \cite{barrett2009finetti,aubrun2022monogamy} we thus have that there is an essentially unique Borel probability measure $\nu$ on $K_A \tmax K_B$ such that
\begin{equation} \label{eq:structureYn-deFinetti}
\S_n(y_n) = \int_{K_A \tmax K_B} w_{AB}^{\otimes n} \dd \nu(w_{AB}).
\end{equation}
Now observe that since $y_n$ is invariant under independent permutations of the copies of $K_A$ or $K_B$, we have that $(1_{A_i} \otimes 1_{B_j})(y_n) = y_{n-1}$ where $1_{A_i}$, $1_{B_j}$ denotes the partial trace over $i$\textsuperscript{th}, $j$\textsuperscript{th} copy of $K_A$, $K_B$, respectively.
It follows that we must have
\begin{equation}
y_n = \otimes_{i=1}^n (1_{A_{n + i}} \otimes 1_{B_{i}})(y_{2n})
\end{equation}
and using \eqref{eq:structureYn-deFinetti} we get
\begin{equation}
y_n = \int_{K_A \tmax K_B} (1_B(w_{AB}))^{\otimes n} \otimes (1_A(w_{AB}))^{\otimes n} \dd \nu(w_{AB}).
\end{equation}
The result follows by defining the map $T: K_A \tmax K_B \to K_A \times K_B$ as $T(w_{AB}) = (1_B(w_{AB}), 1_A(w_{AB}))$, which is continuous and hence measurable, and we get
\begin{equation}
y_n = \int_{K_A \times K_B} (x_A)^{\otimes n} \otimes (x_B)^{\otimes n} \dd T_* \nu (x_A, x_B),
\end{equation}
where $T_* \nu$ is the pushforward measure on $K_A \times K_B$ defined as $T_* \nu (A) = \nu(T^{-1}(A))$ for any measurable set $A \subset K_A \times K_B$, see \cite[Theorem C on page 163]{halmos1950measure} for a proof.
\end{proof}

To formulate the main theorem, we will again use the isomorphism $\S_n: K_A^{\tmax n} \tmax K_B^{\tmax n} \to (K_A \tmax K_B)^{\tmax n}$ defined as $\S_n((\otimes_{i=1}^n x_{i,A}) \otimes (\otimes_{i=1}^n x_{i,B}) ) = \otimes_{i=1}^n (x_{i,A} \otimes x_{i,B})$ for $x_{i,A} \in K_A$ and $y_{i,B} \in K_B$ for all $i \in \{1,\ldots,n\}$.
\begin{theorem} \label{thm:polarization}
Let $K_A$ and $K_B$ be state spaces and let $V$ be a finite-dimensional vector space. Given an affine function $f: K_A \times K_B \to V$ there are $x_A \in K_A$ and $x_B \in K_B$ such that $f(x_A, x_B) = 0$ if and only if there is a bicompatible sequence $\{y_n\}_{n \in \NN}$ such that
\begin{equation} \label{eq:polarization-fconstr}
\Pi((f \otimes f)(\S_2(y_2))) = 0
\end{equation}
where $\Pi: V \otimes V \to W$ is a linear map, where $W$ is a finite-dimensional vector space, such that $\Pi(a \otimes a + b \otimes b) = 0$ implies $a = b = 0$.
\end{theorem}

\begin{proof}
One implication is clear: if there are $x_A \in K_A$ and $x_B \in K_B$ such that $f(x_A, x_B) = 0$, then take $y_n = x_A^{\otimes n} \otimes x_B^{\otimes n}$ and we have
\begin{equation}
\Pi( (f \otimes f)(\S_2(y_2))) = \Pi( f(x_A \otimes x_B) \otimes f(x_A \otimes x_B)) = 0.
\end{equation}
Now assume that $\{y_n\}_{n \in \NN}$ is a bicompatible sequence such that \eqref{eq:polarization-fconstr} holds. The result follows from Lemma~\ref{lemma:structureYn}, we have
\begin{equation} \label{eq:convex_sol}
\S_2(y_2) = \int_{K_A \times K_B} (x_A \otimes x_B)^{\otimes 2} \dd \mu_A(x_A, x_B)
\end{equation}
and we get
\begin{equation} \label{eq:polarization-integral0}
0 = \Pi( (f \otimes f)(\S_2(y_2))) = \int_{K_A \times K_B} \Pi(f^{\otimes 2})(x_A, x_B) \dd \mu_A(x_A, x_B),
\end{equation}
where we use the shorthand $\Pi(f^{\otimes 2})(x_A, x_B) =\Pi(f(x_A, x_B) \otimes f(x_A, x_B))$. The integral above is a well-defined Bochner integral, since $V$ is finite-dimensional and $f$ is affine then $\Pi(f^{\otimes 2})$ is continuous and hence Bochner measurable, and since $K_A \times K_B$ is compact and $\Pi(f^{\otimes 2})$ must be also bounded, $\Pi(f^{\otimes 2})$ is also Bochner integrable. Finally let $W^+$ denote the cone given as $W^+ = \{ \Pi(a \otimes a) : a \in V \}$. Note that $W^+$ is pointed, that is, we have $W^+ \cap (-W^+) = \{ 0 \}$, the proof is as follows: let $w \in W^+ \cap (-W^+)$, then since $w \in W^+$ we have $w = \Pi(a \otimes a)$ but since $w \in - W^+$ we have $w = - \Pi(b \otimes b)$ for some $a,b \in V$. We thus have $0 = w - w = \Pi(a \otimes a + b \otimes b)$ and we get $a = b = 0$ per the assumed properties of $\Pi$. We thus have that $w = 0$ and so $W^+$ is pointed. Finally let $W^{*+}$ denote the dual cone to $W^+$, that is $\varphi \in W^{*+}$ is a functional on $W$ such that $\varphi(\Pi(a \otimes a)) \geq 0$ for any $a \in V$. Using the linearity of the integral in \eqref{eq:polarization-integral0} we get
\begin{equation}
0 = \int_{K_A \times K_B} \varphi(\Pi(f^{\otimes 2})(x_A, x_B)) \dd \mu_A(x_A, x_B)
\end{equation}
and since $\varphi(\Pi(f^{\otimes 2})(x_A, x_B)) \geq 0$ we must have that $\mu$-a.e. $\varphi(\Pi(f^{\otimes 2})(x_A, x_B)) = 0$. Since $W^+$ is pointed, we get that $W^{*+}$ is generating \cite[Proposition B.8]{plavala2023general} and thus $\mu$-a.e. $\psi(\Pi(f^{\otimes 2})(x_A, x_B)) = 0$ for any $\psi \in W^*$. It follows that $\mu$-a.e. $\Pi(f^{\otimes 2})(x_A, x_B) = 0$ and, per our assumption, $f(x_A, x_B) = 0$. This implies that there are $(x_A, x_B) \in K_A \times K_B$ such that $f(x_A, x_B) = 0$.
\end{proof}

There are several possible choices for the map $\Pi: V \otimes V \to W$. For example, if $V = M_\ell(\RR)$ is assumed to be the vector space of $\ell \times \ell$ matrices with real entries, then one can take the Hilbert-Schmidt inner product $\Pi_{\textrm{HS}}(A \otimes B) = \Tr(A^T B)$ where $A^T$ is the transpose of $A$. To verify that $\Pi_{\textrm{HS}}$ satisfies the required property, note that $\Pi_{\textrm{HS}}(A \otimes B) = \Tr(A^T B) = \sum_{i,j=1}^\ell A_{ij} B_{ij}$ where $A_{ij}, B_{ij}$ are the elements of the matrices $A$ and $B$ respectively, and we then get
\begin{equation}
0 = \Pi_{\textrm{HS}}(A \otimes A + B \otimes B) = \sum_{i,j=1}^\ell (A_{ij})^2 + (B_{ij})^2
\end{equation}
which implies $(A_{ij})^2 = (B_{ij})^2 = 0$ and so $A = B = 0$. One can also construct the maps using the matrix product by considering $\Pi_{\textrm{M}}(A \otimes B) = A^T B$, we then get that $0 = \Pi_{\textrm{M}}(A \otimes A + B \otimes B)$ implies
\begin{equation}
0 = \sum_{j=1}^\ell A_{ij} A_{ik} + B_{ij} B_{ik}
\end{equation}
and so $A = B = 0$.
Finally, as already mentioned, one can take the identity map $\id(a \otimes b) = a \otimes b$ where $a,b \in V$ are arbitrary vectors and not necessary matrices; the identity map works since the Kronecker product $a \otimes a$ contains squares of coordinates of $a$.

Note that there are strict relations between $\id$, $\Pi_{\textrm{M}}$, and $\Pi_{\textrm{HS}}$ since $\Pi_{\textrm{M}}(A \otimes B)$ can be computed from $\id(A \otimes B)$ and $\Pi_{\textrm{HS}}(A \otimes B)$ can be computed from $\Pi_{\textrm{M}}(A \otimes B)$. One would thus expect that the hierarchy with $\Pi = \id$ would converge fastest out of the available options; in the following we will present a hierarchy that converges even faster. Once again, we will use the isomorphism $\S_n: K_A^{\tmax n} \tmax K_B^{\tmax n} \to (K_A \tmax K_B)^{\tmax n}$ defined as $\S_n((\otimes_{i=1}^n x_{i,A}) \otimes (\otimes_{i=1}^n x_{i,B}) ) = \otimes_{i=1}^n (x_{i,A} \otimes x_{i,B})$ for $x_{i,A} \in K_A$ and $y_{i,B} \in K_B$ for all $i \in \{1,\ldots,n\}$.
\begin{theorem} \label{thm:polarization++}
Let $K_A$ and $K_B$ be state spaces and let $V$ be a finite-dimensional vector space. Given an affine function $f: K_A \times K_B \to V$ there are $x_A \in K_A$ and $x_B \in K_B$ such that $f(x_A, x_B) = 0$ if and only if there is a bicompatible sequence $\{y_n\}_{n \in \NN}$ such that
\begin{equation} \label{eq:polarization++-fconstr}
(f \otimes \id^{\otimes (n-1)})(\S_n(y_n)) = 0,
\end{equation}
where $f \otimes \id^{\otimes (n-1)}: (A(K_A)^* \otimes A(K_B)^*)^{\otimes n} \to (A(K_A)^* \otimes A(K_B)^*)^{\otimes (n-1)}$.
\end{theorem}
\begin{proof}
The one implication is again straightforward: if there are $x_A \in K_A$ and $x_B \in K_B$ such that $f(x_A, x_B) = 0$, then take $y_n = x_A^{\otimes n} \otimes x_B^{\otimes n}$ and we have
\begin{equation}
(f \otimes \id^{\otimes (n-1)})(\S_n(y_n)) = (f \otimes \id^{\otimes (n-1)})((x_A \otimes x_B)^{\otimes n}) = f(x_A, x_B) (x_A \otimes x_B)^{\otimes (n-1)} = 0.
\end{equation}
Now again assume that $\{y_n\}_{n \in \NN}$ is a bicompatible sequence such that \eqref{eq:polarization++-fconstr} holds.
Using $y_2 = (\id^{\otimes 2} \otimes (1_A \otimes 1_B)^{\otimes (n-2)})(y_n)$ we get
\begin{equation}
(f \otimes f)(y_2) = (f \otimes f \otimes (1_A \otimes 1_B)^{\otimes (n-2)})(y_n) = (f \otimes (1_A \otimes 1_B)^{\otimes (n-2)}) ( (f \otimes \id^{\otimes (n-1)})(\S_n(y_n)) ) = 0
\end{equation}
and so $\{y_n\}_{n \in \NN}$ satisfies \eqref{eq:polarization-fconstr} with $\Pi = \id$ and thus the result follows from Theorem~\ref{thm:polarization}.
\end{proof}

As we will see in the concrete application to non-negative matrix rank, the hierarchy given by Theorem~\ref{thm:polarization++} may converge significantly faster than the hierarchy given by Theorem~\ref{thm:polarization}, meaning that at the same level $n$ the hierarchy given by Theorem~\ref{thm:polarization++} gives at least as good or better results than the hierarchy given by Theorem~\ref{thm:polarization}. This is important since in practical applications we can only solve the hierarchy up to a finite level $n \in \NN$.

\subsection{Generalization}
Using the polarization hierarchy for compact convex sets, it is also possible to optimize over a more general set of problems than those already presented. In particular, let $\vec{\ell} \in \NN^{m}$ and let $f_{\vec{\ell}} : K_{A_1}^{\times \ell_1} \times \ldots \times K_{A_m}^{\times \ell_m} \to V$ and $p: K_{A_1}^{\times \ell_1} \times \ldots \times K_{A_m}^{\times \ell_m} \to \RR$ be affine functions from the product of $\ell_i$ copies of $m$ different state spaces $K_{A_i}$ to $V$ and to $\RR$ respectively. It is then possible to approximate optimization problems of the form
\begin{align}
\begin{split} \label{eq:opt-problem-gen}
p_{\text{gen}}^* = \min_{\{ x_{A_i} \in K_{A_i} \}} \quad & p( \{x_{A_i}^{\times \ell_i}\}_{i=1}^m) \\
\text{s.t.} \quad & f_{\vec{\ell}}(\{x_{A_i}^{\times \ell_i}\}_{i=1}^m)) = 0
\end{split}
\end{align}
Let $\ell = \max\{\ell_i\}$ and
\begin{equation}
y_{n,k} = (\id^{\otimes k} \otimes (\otimes_{i=1}^m 1_{K_{A_i}})^{\otimes (n-k)})(y_n)
\end{equation}
is the restriction of $y_n$ to $k$ copies of each state space. The relaxations of \eqref{eq:opt-problem-gen} corresponding to the generalization of the hierarchy from Theorem~\ref{thm:polarization} is given for $n \geq 2\ell$ by
\begin{align}
\begin{split} \label{eq:pol_hierarchy_gen}
p_{\text{gen}}^n = \min_{y_n \in K_{A_1}^{\tmax n} \tmax \ldots \tmax K_{A_m}^{\tmax n}} \quad & p(y_{n, \ell}) \\
\text{s. t.} \quad & y_n = (\otimes_{i=1}^m S_{\sigma_{A_i}})(y_n) \\
 & \Pi((f_{\vec{\ell}} \otimes f_{\vec{\ell}})(\S_{2\ell}(y_{n, 2\ell})) = 0
\end{split}
\end{align}
and similarly for the generalization of Theorem~\ref{thm:polarization++} by
\begin{align}
\begin{split} \label{eq:pol++_hierarchy_gen}
q_{\text{gen}}^n = \min_{z_n \in K_{A_1}^{\tmax n} \tmax \ldots \tmax K_{A_m}^{\tmax n}} \quad & p(z_{n, \ell}) \\
\text{s. t.} \quad & z_n = (\otimes_{i=1}^m S_{\sigma_{A_i}})(z_n) \\
 & \left( f_{\vec{\ell}} \otimes (\otimes_{i=1}^m \id^{\otimes (n-\ell_i)}) \right) (\S_n(z_n)) = 0.
\end{split}
\end{align}
Here $\S_n$ is an appropriate permutation of the tensor factors. The proof of convergence of these hierarchies then follows from exactly the same reasoning as for the case where $m=2$ and $\ell = 1$.

\subsection{Convex optimization relaxations}
To construct bicompatible sequences, we will use convex relaxations. In particular, if the local state spaces $K_A$ and $K_B$ can be described by inequalities or by LMIs, it is possible to use a linear program (LP) or semidefinite program (SDP) respectively. Consequently, the state spaces $(K_A \tmax K_B)^{\tmax n}$ can be described by an LP or SDP as well. Let $p: K_A \times K_B \to \RR$ and $f: K_A \times K_B \to V$ be affine functions as defined above. As noted before, we want to solve the following optimization problem.
\begin{align}
\begin{split} \label{eq:opt-problem}
p^* = \min_{x_A \otimes x_B \in K_A \tmax K_B} \quad & p(x_A, x_B) \\
\text{s.t.} \quad & f(x_A, x_B) = 0.
\end{split}
\end{align}
Then the relaxations corresponding to Theorem \ref{thm:polarization} are of the form
\begin{align}
\begin{split} \label{eq:pol_hierarchy}
p^n = \min_{y_n \in (K_A \tmax K_B)^{\tmax n}} \quad & p(y_{n,1}) \\
\text{s. t.} \quad & y_n = (S_{\sigma_A} \otimes S_{\sigma_B})(y_n) \\
& \Pi((f\otimes f)(\S_2(y_{n,2})) = 0.
\end{split}
\end{align}
Similarly, the relaxations corresponding to Theorem \ref{thm:polarization++} are of the form
\begin{align}
\begin{split} \label{eq:pol++_hierarchy}
q^n = \min_{z_n \in (K_A \tmax K_B)^{\tmax n}} \quad & p(z_{n,1}) \\
\text{s. t.} \quad & z_n = (S_{\sigma_A} \otimes S_{\sigma_B})(z_n) \\
 & (f \otimes \id^{\otimes (n-1)})(\S_n(z_n)) = 0.
\end{split}
\end{align}
Let $y_n^\star$ (resp.~$z_n^\star)$ be an optimal solution to the optimization problem \eqref{eq:pol_hierarchy} (resp.~\eqref{eq:pol++_hierarchy}), i.e.~a feasible point that achieves $p^n$ (resp.~$q^n)$. By construction, $y_n^\star$ defines a state in $(K_A \tmax K_B)^{\tmax n}$, however, the sequence $(y_n^\star)_n$ obtained from the different levels of the hierarchy will in general not be a bicompatible sequence: each level of the hierarchy will potentially output a completely different solution. The same is true for $(z_n^\star)_n$. The following theorem shows that it is nevertheless possible to construct a bicompatible sequence in the limit.

\begin{theorem} \label{thm:convergence}
It holds that $\lim_{n \to \infty} p^n = p^*$ and $\lim_{n \to \infty} q^n = p^*$.
\end{theorem}

To prove this Theorem, we will first prove Lemma \ref{lem:compact_Ext} below, for which we introduce the following notation. Using the notation of Ref.~\cite{aubrun2022monogamy}, for any state space $K$ we denote by $\gamma_K^{n,k}: (A(K)^*)^{\otimes n} \to (A(K)^*)^{\otimes k}$ the map that symmetrically traces out all but $k$ copies of the ambient vector space of the state space, given by
\begin{align}
\gamma_K^{n,k}
 & = \id_{K}^{\otimes k} \otimes 1_{K}^{\otimes (n-k)} \circ P_{\Sym_n(A(K)^*)},
\end{align}
where
\begin{align}
P_{\Sym_n(A(K)^*)} = \frac{1}{n!} \sum_{\sigma \in S_n} S_\sigma
\end{align}
is the projection onto the symmetric subspace of $(A(K)^*)^{\otimes n}$. A state $x \in K$ is called $n$-extendable if there exists a state $y \in K^{\tmax n}$ such that $x = \gamma_K^{n,1}(y)$. We denote by $\Ext_n(K,k)$ the set of symmetric $k$-th extensions of $n$-extendable states of $K$. That is
\begin{align}
\Ext_n(K,k) := \gamma_K^{n,k}(K^{\tmax n}).
\end{align}

\begin{lemma} \label{lem:compact_Ext}
$\Ext_n(K, k)$ is a compact subset of $K^{\tmax k}$.
\end{lemma}
\begin{proof}
As linear maps between finite dimensional vector spaces, both the maps $P_{\Sym_n(A(K)^*)}$ and $\id^{\otimes k} \otimes 1_K^{\otimes (n-k)}$ are continuous. Since each $K^{\tmax n}$ is compact, it then follows that $\Ext_n(K,k) = \gamma_K^{n,k}(K^{\tmax n})$ is compact as well. Furthermore, since $P_{\Sym_n(A(K)^*)}(K^{\tmax n}) \subset K^{\tmax n}$ and $\id_{K}^{\otimes k} \otimes 1_{K}^{\otimes (n-k)}: K^{\tmax n} \to K^{\tmax k}$, it holds that $\gamma_K^{n,k}(K^{\tmax n})$ is indeed a subset of $K^{\tmax k}$.
\end{proof}

Lastly, we define the set of $k$-th extensions of infinitely extendable states as
\begin{align}
\Ext_\infty(K,k) := \bigcap_{n\geq k} \Ext_n(K,k),
\end{align}
which, as a countable intersection of compact sets, is also compact. We are now ready to prove Theorem \ref{thm:convergence}.

\begin{proof}[Proof (of Theorem \ref{thm:convergence}).]
Since each $p^n$ corresponds to a relaxation, it holds that $p^n \leq p^*$.

For the converse, let $\{y_n^\star\}_{n \in \NN}$ be the sequence of optimal points for the hierarchy of optimization problems \eqref{eq:pol_hierarchy}. For each $n$, the restriction of $y^\star_n$ to two copies of the state spaces $K_A$ and $K_B$ is a state $y^\star_{n,2} \in \Ext_n(K_A \tmax K_B,2) \subset (K_A \tmax K_B)^{\tmax 2}$. By compactness of the state space $(K_A \tmax K_B)^{\tmax 2}$, there is a subsequence of $\{y^\star_{n,2}\}_{n \in \NN}$ that converges to a point $x_2^\star \in (K_A \tmax K_B)^{\tmax 2}$. It then holds that $x_2^\star \in \Ext_\infty(K_A \tmax K_B,2)$. That is, $x_2^\star$ is infinitely extendable and can thus be identified with a bicompatible sequence $\{x_n^\star\}_{n \in \NN}$. Additionally, since each $y_{n,2}^\star$ obeys Eqs.~\eqref{eq:polarization-fconstr}, so does $x_n^\star$. Hence Theorem \ref{thm:polarization} applies and it follows that 
\begin{align}
\S_2(x_2^\star) = \int_{K_A \times K_B} (x_A \otimes x_B)^{\otimes 2} \dd \mu_A(x_A, x_B),
\end{align}
where $x_A$, $x_B$ obey $f(x_A, x_B) = 0$ almost everywhere with respect to $\mu_A$. That is, each such $x_A, x_B$ defines a feasible point of the optimization problem so that $p(x_A, x_B) \geq p^*$ a.e. w.r.t. $\mu_A$.
From this, we conclude that $\lim_n p^n = p(x_2^*) \geq p^\star$.

For the hierarchy \eqref{eq:pol++_hierarchy} it is sufficient to note that the constraint~\eqref{eq:polarization++-fconstr} implies the constraint \eqref{eq:polarization-fconstr}, as was shown in the proof of Theorem \ref{thm:polarization++}. The convergence of \eqref{eq:pol_hierarchy} thus implies convergence of the hierarchy \eqref{eq:pol++_hierarchy}.
\end{proof}

\subsection{Comparison to other methods}
Here, we briefly outline the differences and similarities of the polarization hierarchy with respect to some of the established polynomial optimization hierarchies in the literature. An important difference between the method presented here and most previous methods for polynomial optimization is the use of a de Finetti theorem as a proof technique. Such theorems are commonly seen in quantum information theory papers (see e.g.~Refs.~\cite{doherty2004complete,christandl2007one,brandao2017quantum}), but are used much less in the context of (polynomial) optimization, despite the fact that the first formulation of such a theorem was in terms of (classical) probability distributions \cite{de1929funzione, diaconis1980finite}. In recent years, optimization methods that rely on de Finetti theorems have appeared regularly in the quantum information community \cite{doherty2004complete,navascues2020inflation,wolfe2021quantum,ligthart2022inflation,de2023complete,ligthart2023convergent}.

In commutative polynomial optimization, one of the most well-known methods of solving polynomial optimization problems for compact convex sets is given by the Lasserre hierarchy \cite{lasserre2001global,de2011lasserre,nie2014optimality}. This method systematically provides a convergent hierarchy of semidefinite programs as outer approximations to the polynomial optimization problem one is interested in. More specifically, one constructs a so-called moment matrix and requires it to be positive semidefinite. At finite levels of the hierarchy, this effectively enforces a pseudo probability measure over valid assignments of values to the variables. That is, at level $n$ of the hierarchy the $n$-th order moments of the measure obey all the required constraints.

In contrast, our method does not intrinsically use positive semidefinite constraints. It depends on the shape of the state space which type of hierarchy is implemented: for polyhedra, one gets an LP hierarchy, for spectrahedra an SDP hierarchy. In particular for polyhedra, we thus obtain a convergent LP hierarchy instead of a convergent SDP hierarchy. This has the advantage that LPs can be implemented for a larger number of variables and constraints than SDPs, and higher levels of the hierarchy can therefore potentially be reached. State spaces that are not described by linear matrix inequalities, however, are in general more difficult to implement using our hierarchy.

The generalization of the Lasserre hierarchy to non-commutative polynomials was derived in Ref.~\cite{pironio2010convergent} and is known as the NPA hierarchy. This hierarchy was constructed to allow for optimization over states and representations of $C^*$-algebras. As a consequence the dimension of the quantum states is not fixed. Importantly, the NPA hierarchy only allows for expressions that are polynomial in the operators, but linear in the expectation values of such operators. Recently, two hierarchies have been proposed to also allow for optimization over polynomials in the expectation values: The state polynomial optimization hierarchy \cite{klep2023state}, and the quantum polarization hierarchy that this paper is in part inspired by \cite{ligthart2022inflation,ligthart2023convergent}. The approach presented in this paper differs in the sense that one optimizes explicitly over a finite-dimensional state space, and does not regard an algebra of operators directly. Polynomials of expectation values in observables enter the description as functions on the finite-dimensional state space through the polynomials $p$ and $f$. Thus, instead of optimizing over states on a $C^*$-algebra, the optimization runs over all states of a fixed (finite) dimensional state space.

A second method that is commonly used to approximate polynomial optimization problems is the Sherali-Adams hierarchy \cite{sherali1990hierarchy, sherali1992global}. This hierarchy makes use of the so-called Reformulation-Linearisation Technique (RLT) to build a hierarchy of LP approximations, which is very closely related to the polarization technique in the case of polyhedra. Using the notation of this paper, the original S-A hierarchy is specialized to optimization problems of the form
\begin{align}
\min_{x \in \RR^N} \quad & p(x) \\
 & f(x) \geq 0 \\
 & 0 \leq l_j \leq x_j \leq u_j < \infty. \label{eq:sherali-adams-cone-constraints}
\end{align}
The constraint \eqref{eq:sherali-adams-cone-constraints} can be interpreted in our language to define a polyhedral state space. At level $n$ of the hierarchy of relaxations, one then defines new variables $X_J$ for $J \subset [N]$ for all subsets with repetitions $J$ of size $n$, which are intended to mimic the monomials $\prod_{i \in J} x_i$. Additionally, one takes all the products of the constraints \eqref{eq:sherali-adams-cone-constraints}, which is equivalent to taking the maximal tensor product of the state space. At high enough level of the hierarchy it then becomes possible to implement the polynomial constraints $f(x) \geq 0$ and to optimize over $p$, in a similar way as for the polarization hierarchy. It is then proven that this hierarchy converges to the \emph{convex hull relaxation} of the original optimization problem, in which one only requires that a convex combination of assignments to the variables fulfill all the constraints (compare with e.g.~Eq.~\eqref{eq:convex_sol}, where, a priori, a similar result is obtained, but is later resolved). The notable difference is that, in the polarization hierarchy presented here, the polynomial constraints are squared, which guarantees convergence to the optimal value of the original optimization problem, and not to its convex hull relaxation. Additionally, the polarization hierarchy is not restricted to state spaces of the form \eqref{eq:sherali-adams-cone-constraints}, though it should be noted that there have been many adaptations of the S-A hierarchy, which also include semidefinite optimization problems \cite{sherali2002enhancing} and non-linear optimization problems \cite{zhen2021extension}.

From a slightly different viewpoint, the polarization hierarchy can also be seen as an adaptation and generalization of the Doherty-Parrilo-Spedalieri (DPS) hierarchy \cite{doherty2004complete,aubrun2022monogamy} to more general optimization problems. The DPS hierarchy is often used in quantum information theory to provide a relaxation to the set of \emph{separable quantum states}, i.e.~convex mixtures of product states in finite dimensional quantum theory. The polarization hierarchy instead provides the opportunity to optimize over approximations to product states, as opposed to separable states; the same advantage that arises in the comparison with the S-A hierarchy. Additionally, the DPS hierarchy is most often used in combination with optimization problems that are linear in the state (e.g.~in \cite{brandao2005quantifying,cavalcanti2015detection,de2023complete}), while the polarization hierarchy allows for general polynomial expressions.

Finally, it should be noted that the polarization hierarchy currently does not include polynomial \emph{inequality} constraints, while both the Lasserre and Sherali-Adams hierarchies do. It is not difficult to construct relaxations that also include such inequality constraints, though it is not clear whether these relaxations converge to the correct optimal value. Similar to the Sherali-Adams hierarchy, however, convergence to the convex hull relaxation of the problem follows straightforwardly from the proofs of Theorems \ref{thm:polarization} and \ref{thm:convergence}. We leave a construction and proof of convergence for a hierarchy with inequalities open for future research.

\section{Application to nonnegative matrix rank and nested rectangles problem}
The nonnegative matrix rank is used in combinatorial optimization, but also in quantum theory and statistics and other applications, see \cite{cohen1993nonnegative} for a review. First of all, note that $A$ and also $L$ and $R$ can, without loss of generality, be considered to be left-stochastic matrices, which is important since the set of left-stochastic matrices is compact and convex. We call a matrix $X$ left-stochastic if $X_{ij} \geq 0$ for all $i,j$ and $\sum_i X_{ij} = 1$ for all $j$. Without loss of generality assume that $A$ does not contain any rows and columns where all entries would be identically zero. There is a diagonal $m \times m$ matrix $D_A$ with non-negative entries that normalizes the columns of $A$, that is, such that $\tilde{A} = A D_A$ is left-stochastic, and, analogously, there is a diagonal $k \times k$ matrix $D_L$ with non-negative entries that normalizes the columns of $L$, that is, such that $\tilde{L} = L D_L$ is left-stochastic. We get $\tilde{A} = \tilde{L} D_L^{-1} R D_A$ and denoting $\tilde{R} = D_L^{-1} R D_A$ we get the factorization $\tilde{A} = \tilde{L} \tilde{R}$, note that $\tilde{R}$ has the same dimensions as $R$. It follows that $\tilde{R}$ is left-stochastic as we have $1 = \sum_{i=1}^n \tilde{A}_{ij} = \sum_{i=1}^n \sum_{\ell=1}^k \tilde{L}_{i \ell} \tilde{R}_{\ell j} = \sum_{\ell=1}^k \tilde{R}_{\ell j}$ where we have used both the left-stochasticity of $\tilde{A}$ and $\tilde{L}$. Moreover, if $\tilde{A} = L' R'$ is any factorization of $\tilde{L}$, then $A = L' (R' D_A^{-1})$ is a factorization of $A$.

If we view the matrices $L$ as being part of a state space $K_A$ of left-stochastic matrices of dimensions $n \times k$, and similarly consider the matrices $R$ as part of a state space $K_B$ of left-stochastic matrices of dimensions $k \times m$, it is readily seen that the matrix multiplication $A = LR$ can be thought of as an affine map $f: K_A \times K_B \to V$. A convenient toy example for the nonnegative matrix rank problem is given by the nested rectangles problem \cite{gribling2019lower}. It asks, given a square $P = [-1, 1]^2$, for which values of $a$ and $b$ there exists a triangle $T$, such that
\begin{align*}
[-a, a] \times [-b, b] \subseteq T \subseteq P.
\end{align*}
It turns out that this question is equivalent to determining whether the matrix given by
\begin{align} \label{eq:Mdef}
M = \frac{1}{4}
\begin{pmatrix}
1-a & 1+a & 1-b & 1+b \\
1+a & 1-a & 1-b & 1+b \\
1+a & 1-a & 1+b & 1-b \\
1-a & 1+a & 1+b & 1-b
\end{pmatrix}
\end{align}
has nonnegative rank equal to 3 \cite{gribling2019lower}. It is known that such a nested triangle $T$ exists if and only if $(1+a)(1+b) \leq 2$ \cite{fawzi2015lower}. This problem can therefore give an indication for the quality of a bound provided by a relaxation, by comparing it to this known analytic bound.

We aim to show the (non-)existence of two element-wise nonnegative, left-stochastic matrices $U \in \mathbbm{M}^{4 \times 3}$ and $V \in \mathbbm{M}^{3 \times 4}$ such that $M = UV$. The matrix elements $U_{ij}, V_{ij}$ will form the bases of the state spaces. To every state space we will add a ``unit'' dimension, denoted by $\Id$, which will signal when the system has been traced out. Additionally, we will use the left-stochasticity to remove a dimension for every column in the matrices. The effective local dimensions of $K_A$ and $K_B$ are therefore $d_A=12 + 1 - 3 =10$ and $d_B = 12 + 1 - 4 = 9$ respectively and the equality constraints for left-stochasticity are replaced by inequalities of the form
\begin{align}
\Id - \sum_{i=1}^{3} U_{ij} & \geq 0, \label{eq:Ustoc} \\
\Id - \sum_{i=1}^{2} V_{ij} & \geq 0. \label{eq:Vstoc}
\end{align}

Subsequently, we construct variables $y^n_{w}$, labeled in the words constructed from the basis elements of $K_A \otimes K_B$ of length $\leq n$ at level $n$ of the hierarchy, with $\Id$ corresponding to the empty word. For example, for level $n=2$ of the hierarchy one such variable could be $y^2_{U_{ij} V_{kl} \Id \Id} =: y^2_{U_{ij} V_{kl}}$. In order to significantly reduce the size of the LP hierarchy at every step, we will encode the problem on the symmetric subspace. That is, we directly impose the symmetry constraints of Eq.~\eqref{eq:symmetry} to reduce the number of variables to $\binom{10 + n -1}{n} \cdot \binom{9 + n - 1}{n}$, as opposed to the $(10 \cdot 9)^n$ variables that follow from the tensor product description. This has the added benefit of removing all the symmetry constraints from the LP description as well. We therefore only have to look at the variables $y^n_{[w]}$, where $[w]$ denotes the equivalence class of the word $w$ under the symmetry conditions. The LP implementing the hierarchy of Eq.~\eqref{eq:pol++_hierarchy} then has the following form
\begin{align}
\begin{split} \label{eq:rectangle_LP}
\min_{y^n_{[w]}}\quad & 0 \\
\text{s. t.} \qquad & y^n_{[\Id]} = 1 \\
 & 0 \leq y^n_{[w]} \leq 1 \quad \forall [w]: \abs{[w]} \leq 2n \\
 & y^n_{[w]} - \sum_{i=1}^3 y^n_{[U_{ij} w]} \geq 0 \quad \forall [w]: \abs{[w]} \leq 2n-1 \\
 & y^n_{[w]} - \sum_{i=1}^2 y^n_{[V_{ij} w]} \geq 0 \quad \forall [w]: \abs{[w]} \leq 2n-1 \\
 & M_{ij} \cdot y^n_{[w]} = \sum_{j=1}^3 y^n_{[U_{ij} V_{jk} w]} \quad \forall [w]: \abs{[w]} \leq 2n-2, \ \forall i,j \in [4].
\end{split}
\end{align}
Note that the last constraint implicitly rewrites the removed variables $U_{4j}$ and $V_{3k}$ into their canonical forms using Eqs.~\eqref{eq:Ustoc} and \eqref{eq:Vstoc}.

Infeasibility of this LP witnesses incompatibility of the matrix $M$ with a rank-3 decomposition into non-negative matrices, and therefore answers the nested rectangle problem in the negative for those values of $a$ and $b$. We have implemented this program for $n=3$ and compared it with existing methods \cite{gribling2019lower, fawzi2015lower, fawzi2016self}. The results can be seen in Fig.~\ref{fig:nested_rectangles}, the code used to generate this data is publically available \cite{laurens_github}. By running the LP for various values of $a$ and $b$ and using a bisection method, we determined the region in which the hierarchy \eqref{eq:rectangle_LP} can detect infeasibility at level $n=3$, which is denoted in Fig.~\ref{fig:nested_rectangles} by the blue area. We also implemented the weaker hierarchy corresponding to Theorem~\ref{thm:polarization}, which could detect infeasibility for all values of $a$ and $b$ in the red area. 
Our method outperforms the ones of Refs.~\cite{gribling2019lower, fawzi2016self}, depicted in \cite[Fig.~1]{gribling2019lower}, indicating relatively fast convergence for this particular problem. It should, however, be noted that the polarization hierarchy quickly becomes impractical for state spaces with larger dimensions and for higher levels of the hierarchy.

\begin{figure}
\centering
\includegraphics[width=0.5\linewidth]{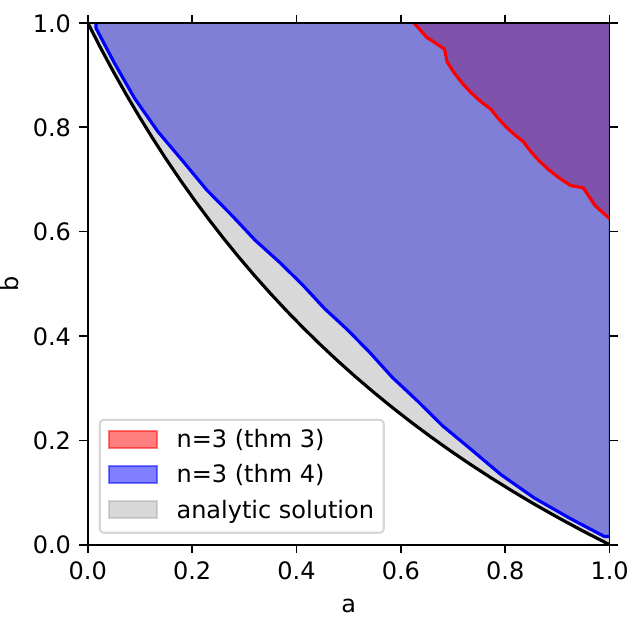}
\caption{Application of the polarization hierarchies given by Theorem~\ref{thm:polarization} and Theorem~\ref{thm:polarization++} for $n=3$ to the nested rectangles problem. It is known analytically that the matrix $M$ given by \eqref{eq:Mdef} does not have rank-3 decomposition within the gray region. This problem is used to compare the convergence of the polarization hierarchy to other methods. The red region is where rank-3 decomposition is ruled out by the hierarchy given by Theorem~\ref{thm:polarization} for $n=3$ and the blue region is where a rank-3 decomposition is ruled out by the hierarchy given by Theorem~\ref{thm:polarization++} for $n=3$.}
\label{fig:nested_rectangles}
\end{figure}

\begin{acknowledgments}
We are thankful to Sander Gribling for discussions.
MP acknowledges support from the Deutsche Forschungsgemeinschaft (DFG, German Research Foundation, project numbers 447948357 and 440958198), the Sino-German Center for Research Promotion (Project M-0294), the German Ministry of Education and Research (Project QuKuK, BMBF Grant No. 16KIS1618K), the DAAD, and the Alexander von Humboldt Foundation.
The work of LL and DG has been supported by
the Bundesministerium für Bildung und Forschung -- BMBF under projects QuBRA and ProvideQ, as well as
Germany's Excellence Strategy – Cluster of Excellence Matter and Light for Quantum Computing (ML4Q) EXC 2004/1 (390534769).
\end{acknowledgments}

\bibliography{citations}

\end{document}